\documentclass[10pt]{article}

\usepackage{amssymb,latexsym}
\usepackage{amsfonts}
\usepackage{amsmath,amsthm,graphicx}

\newcounter{cnt1}
\newcounter{cnt2}
\newcounter{cnt3}
\newcommand{\blr}{\begin{list}{$($\roman{cnt1}$)$} {\usecounter{cnt1}
        \setlength{\topsep}{0pt} \setlength{\itemsep}{0pt}}}
\newcommand{\bla}{\begin{list}{$($\alph{cnt2}$)$} {\usecounter{cnt2}
       \setlength{\topsep}{0pt} \setlength{\itemsep}{0pt}}}
\newcommand{\bln}{\begin{list}{$($\arabic{cnt3}$)$} {\usecounter{cnt3}
                \setlength{\topsep}{0pt} \setlength{\itemsep}{0pt}}}
\newcommand{\el}{\end{list}}

\newtheorem{thm}{Theorem}[section]

\newtheorem{prop}[thm]{Proposition}
\newtheorem{defn}[thm]{Definition}
\newtheorem{Exm}[thm]{Example}
\newtheorem{rem}[thm]{Remark}

\newtheorem{cor}[thm]{Corollary}

\newcommand{\ilim}{\mathop{\varprojlim}\limits}

\begin{document}
\title {An Example of Constructing Versal Deformation for Leibniz Algebras}
\author{Ashis Mandal
}
\maketitle
\date{}
\noindent
\begin{abstract}
  In this work we compute a versal deformation of the three dimensional nilpotent Leibniz algebra over $\mathbb{C}$, defined by the nontrivial brackets $[e_1,e_3]=e_2$ and $[e_3,e_3]=e_1$.  
\end{abstract}
{\bf Keywords:} Leibniz algebra, Leibniz cohomology, infinitesimal deformation, versal deformation, obstruction.\\
 {\bf Mathematics Subject Classifications (2000):} $13$D$10$, $14$D$15$, $13$D$03$.
\section{Introduction}
Leibniz algebras are a generalized version of Lie algebras, without the antisymmetry property. They were introduced by J.-L. Loday in 1993 and they turned out to be useful both in mathematics and physics. In \cite{FMM} the authors develop the versal deformation theory for Leibniz algebras. The existence of a versal deformation under certain cohomology condition follows from a general theorem of Schlessinger \cite{Sch}. The construction of a versal deformation is essential to solve the basic deformation question, as it is a deformation which induces all nonequivalent deformations of a given Leibniz algebra. 

In this paper we give an explicit example on which we demonstrate the general construction and computations. For this, after recalling some definitions and results in Section \ref{lcohomology}, we describe and prove the relationship between Massey brackets and obstructions for Leibniz algebra deformations in Section \ref{Massey Brackets and Obstructions}.

Our example is the following.  
Consider a three dimensional vector space $L$ spanned by $\{e_1,~e_2,~e_3\}$ over $\mathbb{C}$. Define a bilinear map $[~,~]: L\times L \longrightarrow L$ by $[e_1,e_3]=e_2$ and $[e_3,e_3]= e_1$, all other products of basis elements being $0$. Then $(L,[~,~])$ is a Leibniz algebra over $\mathbb{C}$ of dimension $3$. The Leibniz algebra $L$ is  nilpotent and is denoted by $\lambda_6$  in the classification of three dimensional nilpotent Leibniz algebras, see \cite{A3}. We compute cohomologies necessary for our purpose, Massey brackets and construct a versal deformation of our example in Section \ref{computation}. 
\section{Leibniz Algebra, Cohomology and Deformations} \label{lcohomology}
Leibniz algebras were introduced by J.L.-Loday \cite{L1,L3} and their cohomology was defined in \cite{LP,L2}. Let us recall some basic definitions. Let $\mathbb{K}$ be a field.
\begin{defn}
A Leibniz algebra is a $\mathbb{K}$-module $L$, equipped with a bracket operation that satisfies the Leibniz identity: 
$$[x,[y,z]]= [[x,y],z]-[[x,z],y],~~\mbox{for}~x,~y,~z \in L.$$ 
\end{defn}

Any Lie algebra is automatically a Leibniz algebra, as in the presence of antisymmetry, the Jacobi identity is equivalent  to the Leibniz identity. More  examples of Leibniz algebras were given in \cite {L1,LP}, and recently for instance in \cite{A3,A1, A2}.

Let $L$ be a Leibniz algebra and $M$ a representation of $L$. By definition, $M$ is a $\mathbb{K}$-module equipped with two actions (left and right) of $L$,
$$[-,-]:L\times M\longrightarrow M~~\mbox{and}~[-,-]:M \times L \longrightarrow M ~~\mbox{such that}~$$ 
$$[x,[y,z]]=[[x,y],z]-[[x,z],y]$$
holds, whenever one of the variables is from $M$ and the two others from $L$.

Define $CL^n({L}; {M}):= \mbox{Hom} _\mathbb{K}({L}^{\otimes n}, {M}), ~n\geq 0.$ Let 
$$\delta^n : CL^n({L}; {M})\longrightarrow CL^{n+1}(L; M)$$ 
be a $\mathbb{K}$-homomorphism defined by 
\begin{equation*}
\begin{split}
&\delta^nf(x_1,\cdots,x_{n+1})\\
&:= [x_1,f(x_2,\cdots,x_{n+1})] + \sum_{i=2}^{n+1}(-1)^i[f(x_1,\cdots,\hat{x}_i,\cdots,x_{n+1}),x_i]\\
&+ \sum_{1\leq i<j\leq n+1}(-1)^{j+1}f(x_1,\cdots,x_{i-1},[x_i,x_j],x_{i+1},\cdots,\hat{x}_j,\cdots, x_{n+1}).
\end{split}
\end{equation*}
Then $(CL^*(L; M),\delta)$ is a cochain complex, whose cohomology is called the cohomology of the Leibniz algebra $L$ with coefficients in the representation $M$. The $n$ th cohomology is denoted by $HL^n(L; M)$. In particular, $L$ is a representation of itself with the obvious action given by the bracket in $L$. The  $n$ th cohomology of $L$ with coefficients in itself is denoted by $HL^n(L; L).$
Let $S_n$ be the symmetric group of $n$ symbols. Recall that a permutation $\sigma \in S_{p+q}$ is called a $(p,q)$-shuffle, if $\sigma(1)<\sigma(2)<\cdots<\sigma(p)$, and $\sigma(p+1)<\sigma(p+2)<\cdots<\sigma(p+q)$. We denote the set of all $(p,q)$-shuffles in $S_{p+q}$ by $Sh(p,q)$.

For $\alpha \in CL^{p+1}(L;L)$ and $\beta \in CL^{q+1}(L;L)$, define $\alpha \circ \beta \in CL^{p+q+1}(L;L)$ by
\begin{equation*}
\begin{split}
&\alpha \circ \beta (x_1,\ldots,x_{p+q+1} )\\
=&~\sum_{k=1}^{p+1}(-1)^{q(k-1)}\{\sum_{\sigma \in Sh(q,p-k+1)}sgn(\sigma)\alpha(x_1,\ldots,x_{k-1},\beta(x_k,x_{\sigma(k+1)},\ldots,x_{\sigma(k+q)}),\\
&~~~~~~~~~~~~~~~~~~~~~~~~~~~~~~~~~~~~~~~~~~~~~~~~~~~~~x_{\sigma(k+q+1)},\ldots,x_{\sigma(p+q+1)}) \}.
\end{split}
\end{equation*}

The graded cochain module 
$CL^{*}(L;L)=\bigoplus_{p} CL^p(L;L)$ equipped with the bracket $\nu$ as defined by  
$$[\alpha,\beta]=\alpha \circ \beta + (-1)^{pq+1} \beta \circ \alpha
~~\mbox{for}~ \alpha \in CL^{p+1}(L;L)~~\mbox{and}~\beta \in CL^{q+1}(L;L)$$
 and the differential map $d$  by $d \alpha =(-1)^{|\alpha|}\delta \alpha~\mbox{for}~\alpha \in CL^{*}(L;L) $ is a differential graded Lie algebra \cite{B}. 

Let now $\mathbb{K}$ a field of zero characteristic and the tensor product over $\mathbb{K}$ will be denoted by $\otimes$. We recall the notion of deformation of a Leibniz algebra $L$ over a local algebra base $A$ with a fixed augmentation $\varepsilon:{A}\rightarrow
\mathbb{K}$ and maximal ideal $\mathfrak{M}$. Assume $dim(\mathfrak{M}^k/\mathfrak{M}^{k+1})<\infty$ for every $k$ (see \cite{FMM}).
\begin{defn}
A deformation $\lambda$ of ${L}$ with base
$({A},\mathfrak{M})$, or simply with base ${A}$ is an ${A}$-Leibniz
algebra structure on the tensor product
${A}\otimes {L}$ with the bracket $[,]_\lambda$ such that
 \[
 \varepsilon\otimes id:{A}\otimes {L}\rightarrow \mathbb{K}\otimes {L}
 \]
 is a ${A}$-Leibniz algebra  homomorphism (where the $A$-Leibniz algebra structure on $\mathbb{K}\otimes {L}$ is given via $\varepsilon$).
\end{defn}
A deformation of the Leibniz algebra $L$ with base $A$ is called {\it infinitesimal}, {or \it first order}, if in addition to this $\mathfrak{M}^2=0$. We call a deformation of {\it order k}, if $\mathfrak{M}^{k+1}=0$.

Suppose $A$ is a complete local algebra ( $A=\ilim_{n\rightarrow
\infty}({A}/{\mathfrak{M}^n})$), where $\mathfrak{M}$ is the maximal
ideal in $A$. Then a deformation of $L$ with base $A$ which is obtained as the projective limit of deformations of $L$ with base $A/\mathfrak{M}^{n}$ is called a {\it formal deformation} of $L$.

Observe that for $l_1,l_2 \in L$ and $a,b \in A$ we have $$[a\otimes l_1,b\otimes l_2]_\lambda = ab[1\otimes l_1,1\otimes l_2]_\lambda$$ by $A$- linearity of $[,]_\lambda$.
Thus to define  a deformation  $\lambda$ it is enough to specify the brackets $[1\otimes l_1,1\otimes l_2]_\lambda$ for $l_1,l_2 \in L$.
Moreover, since $\varepsilon\otimes id:{A}\otimes {L}\rightarrow \mathbb{K}\otimes {L}$
 is a ${A}$-Leibniz algebra  homomorphism, 
 $$(\varepsilon\otimes id)[1\otimes l_1,1\otimes l_2]_\lambda =[l_1,l_2]=(\varepsilon\otimes id)(1\otimes[l_1,l_2])$$
which implies $$ [1\otimes l_1,1\otimes l_2]_\lambda -1\otimes[l_1,l_2] \in ker (\varepsilon \otimes id).$$
Hence we can write 
$$[1\otimes l_1,1\otimes l_2]_\lambda =1\otimes[l_1,l_2]+\sum_{j} c_j \otimes y_j ,$$ where $\sum_{j} c_j \otimes y_j$ is a finite sum with $c_j \in ker(\varepsilon)=\mathfrak{M}$ and $y_j \in L$.

\begin{defn}
Suppose $\lambda_1$ and $\lambda_2$ are two deformations of a Leibniz algebra $L$ with finite dimensional local algebra base $A$. We call them  equivalent if  there exists a Leibniz algebra isomorphism $$ \phi:(A\otimes L,[,]_{\lambda_1})\rightarrow (A\otimes L,[,]_{\lambda_2})$$ such that $(\varepsilon\otimes id)\circ \phi=\varepsilon\otimes id$.
\end{defn}
The definition naturally generalizes to deformations  complete local algebra base.
We write $\lambda_1 \cong \lambda_2$ if $\lambda_1$ is equivalent to $\lambda_2$. 
\begin{Exm}
If $A= \mathbb{K}[[t]]$ then a formal deformation of a Leibniz algebra $L$ over $A$ is precisely a formal $1$-parameter deformation of $L$(see \cite{B}).
\end{Exm}
\begin{defn}
Suppose $\lambda$ is a given deformation of $L$ with base $(A,\mathfrak{M})$ and augmentation $\varepsilon:{A}\rightarrow \mathbb{K}$, where $A$ is a finite dimensional local algebra. Let $A^\prime$ be another commutative local algebra with identity and augmentation  $\varepsilon^{\prime}:{A^\prime}\rightarrow \mathbb{K}$. Suppose $\phi:A \rightarrow A^{\prime} $ is an algebra homomorphism with $\phi(1)=1$ and $\varepsilon^{\prime} \circ \phi =\varepsilon$. Let $ker(\varepsilon^{\prime})= \mathfrak{M}^\prime$. Then the push-out $\bf{\phi_{*} \lambda}$ is the deformation of $L$ with base $(A^\prime,\mathfrak{M}^\prime)$ and bracket 
          $$[{a_1 }^\prime \otimes_A (a_1\otimes {l_1}),a_2 ^\prime
\otimes_A(a_2\otimes l_2) ]_{\phi_* \lambda}=a_1 ^\prime a_2 ^\prime
\otimes_A[a_1\otimes l_1,a_2\otimes l_2]_\lambda $$
 where $a_1^\prime,a_2 ^\prime \in {A}^\prime,~ a_1,a_2 \in A$ and $l_1,l_2 \in
L$. Here $A^\prime$ is considered as an $A$-module by the map $a^\prime \cdot a=a^\prime \phi(a)$ so that $$A^\prime \otimes L=(A^\prime {\otimes}_{A} A)\otimes L =A^\prime {\otimes}_{A}(A \otimes L).$$
\end{defn}
The same definition holds for complete algebra base by taking projective limit.

\begin{rem}\label{push-out}
If the bracket $[,]_\lambda$ is given by 
\begin{equation*}
[1\otimes l_1,1\otimes l_2]_\lambda =1\otimes[l_1,l_2]+\sum_{j} c_j \otimes y_j~\mbox{for}~ c_j \in \mathfrak{M}~\mbox{and}~y_j \in L
\end{equation*} 
then the bracket $[,]_{\phi_* \lambda}$ can be written as
\begin{equation*}
[1\otimes l_1,1\otimes l_2]_ {\phi_* \lambda}=1\otimes [l_1,l_2]+\sum_{j}\phi(c_j) \otimes y_j.
\end{equation*}
\end{rem}
Let us recall the construction of a specific infinitesimal deformation of a Leibniz algebra $L$, which is universal in the class of all infinitesimal deformations from \cite{FMM}.
Assume that $dim (HL^2(L;L)) < \infty$. Denote the space $HL^2(L;L)$ by $\mathbb{H}$. Consider the algebra $C_1=\mathbb{K}\oplus \mathbb{H}^\prime$  where $\mathbb{H}^\prime$ is the dual of $\mathbb{H}$~, by setting $$(k_1,h_1)\cdot(k_2,h_2)=(k_1 k_2,k_1 h_2+k_2 h_1)~\mbox{for}~(k_1,h_1), (k_2,h_2)\in C_1.$$   Observe that the second summand is an ideal of $C_1$ with zero multiplication. Fix a homomorphism $$\mu: \mathbb{H} \longrightarrow
CL^2(L;L)=Hom(L^{\otimes 2};L) $$
which takes a cohomology class into a cocycle representing it. Notice that there is an isomorphism $\mathbb{H}^\prime \otimes L \cong Hom(\mathbb{H}~;L)$, 
so we have
$$C_1 \otimes L 
 =(\mathbb{K} \oplus \mathbb{H}^\prime)\otimes L 
 \cong (\mathbb{K}\otimes L) \oplus (\mathbb{H}^\prime \otimes L) 
 \cong L \oplus Hom(\mathbb{H}~;L).$$
Using the above identification, define a Leibniz bracket on $C_1 \otimes L$ as follows.
For $(l_1,\phi_1),(l_2,\phi_2) \in L \oplus Hom(\mathbb{H}~;L)$ let
$$[(l_1,\phi_1),(l_2,\phi_2)]=([l_1,l_2],\psi)$$ where the map  $\psi:\mathbb{H} \longrightarrow L$ is given by 
$$\psi(\alpha)=\mu(\alpha)(l_1,l_2)+[\phi_{1}(\alpha),l_2]+[l_1,\phi_2(\alpha)]~\mbox{for}~\alpha \in \mathbb{H}~.$$ 
It is straightforward to check that $C_1\otimes L$ along with the above bracket is a Leibniz algebra over $C_1$. The Leibniz identity is a consequence of the fact that $\delta \mu(\alpha)=0~ \mbox{for}~ \alpha \in \mathbb{H}$~. Thus $\eta_1$ is an infinitesimal deformation of $L$ with base $C_1=\mathbb{K}\oplus \mathbb{H}^\prime $. It is proved in \cite{FMM}:
\begin{prop}\label{up toisomorphism}
Up to an isomorphism, the deformation $\eta_1$ does not depend on
the choice of $\mu$. 
\end{prop} 
  
\begin{rem}\label{exp of inf}
Suppose $\{h_i \}_{1\leq i \leq n}$ is a basis of $\mathbb{H}$ and $\{g_i\}_{1\leq i \leq n}$ is the dual basis. Let $\mu(h_i)=\mu_i \in CL^2(L;L)$. Under the identification $C_1 \otimes L = L \oplus Hom(\mathbb{H}~;L)$, an element $(l,\phi)\in L \oplus Hom(\mathbb{H}~;L)$ corresponds to $1\otimes l +\sum_{i=1}^{n}{g_i\otimes \phi(h_i)}$. 
Then for $(l_1,\phi_1),(l_2,\phi_2) \in L \oplus Hom(\mathbb{H};L)$ their bracket $([l_1,l_2],\psi)$ 
corresponds to 
$$1\otimes [l_1,l_2]+ \sum_{i=1}^{n} g_i\otimes (\mu_i(l_1,l_2)+[\phi_1(h_i),l_2]+[l_1,\phi_2(h_i)]).$$
In particular, for $l_1,l_2 \in L$ we have 
$$[1\otimes l_1,1\otimes l_2]_{\eta_1}=1\otimes [l_1,l_2]+\sum_{i=1}^{n}g_i \otimes \mu_i(l_1,l_2).$$ 
\end{rem}
The main property of $\eta_{1}$ is the universality in the class of infinitesimal deformations with a finite dimensional base. 
 
\begin{prop}\label{couniversal}
For any infinitesimal deformation $\lambda$ of a Leibniz algebra
$L$ with a finite dimensional  base $A$ there exists a unique
homomorphism $\phi:C_1=({\mathbb{K}}\oplus
\mathbb{H}^\prime)\longrightarrow A$ such that $\lambda$ is
equivalent to the push-out $\phi_{*}\eta_1$.
\end{prop}
Suppose $A$ is a local algebra with the unique maximal ideal $\mathfrak{M}$ and $\pi:A\rightarrow A/{\mathfrak{M}^{2}}$ the corresponding quotient map. The algebra $A/{\mathfrak{M}^{2}}$ is obviously local with maximal ideal ${\mathfrak{M}}/{\mathfrak{M}^{2}}$ and $({\mathfrak{M}}/{\mathfrak{M}^{2}})^{2}=0$. If $\lambda$ is a deformation of $L$ with base $A$ then $\pi_{*} \lambda$ is a deformation with base $A/{\mathfrak{M}^{2}}$ and it is clearly infinitesimal. Therefore, by the previous proposition, we have a map 
$$a_{\pi *\lambda}:({\mathfrak{M}}/{\mathfrak{M}^{2}})^\prime \rightarrow \mathbb{H}~. $$
\begin{defn}

The dual space $({\mathfrak{M}}/{\mathfrak{M}^{2}})^\prime$  is called the tangent space of A and is denoted by $TA$. The map $a_{\pi *\lambda}$ is called the differential of $\lambda$ and is denoted by $d{\lambda}$.
\end{defn}
It follows from Proposition \ref{couniversal} that equivalent deformations have the same differential (see \cite{FMM}). 
\begin{defn}
Let $C$ be a complete local algebra. A formal deformation $\eta$ of a Leibniz algebra $L$ with base $C$ is called versal, if\\
(i)~for any formal deformation $\lambda$ of $L$ with  base $A$ there exists a homomorphism $f:C \rightarrow A$ such that the deformation $\lambda$ is equivalent to $f_{*}\eta$; \\
(ii)~if $A$ satisfies the condition ${\mathfrak{M}}^2=0$, then $f$ is unique. 
\end{defn}  

In \cite{FMM} a construction for a versal deformation of a Leibniz algebra was given. The construction involves realizing obstructions to extend a deformation  with base $A$ to a deformation  with base $B$ for a given extension 
$$0\longrightarrow
{M}\stackrel{i}{\longrightarrow} B
\stackrel{p}{\longrightarrow}A\longrightarrow 0.$$
Suppose a deformation $\lambda$ of $L$ is given with base $A$. If we try to extend it to a deformation with base $B$, it gives rise to a cohomology class in 
$$ HL^3(L;M\otimes L)=M\otimes HL^3(L;L).$$
The above assignment yields the {\it obstruction map} for this extension 
$$ \theta_{\lambda}:H_{Harr}^2(A;M) \longrightarrow M \otimes HL^3(L;L),~(\mbox{see}~\cite{FMM}).~$$
(Here $H_{Harr}^2(C_k;\mathbb{K})$ denotes the two dimensional Harrison cohomology space.)

Let us recall the main steps of the construction.
Consider the Leibniz algebra $L$ with $dim(\mathbb{H})<\infty$ and the extension
$$0\longrightarrow
 \mathbb{H}^\prime \stackrel{i}{\longrightarrow} C_1
\stackrel{p}{\longrightarrow} C_0 \longrightarrow 0,$$ 
where  $C_0=\mathbb{K}$ and 
$C_1=\mathbb{K}\oplus \mathbb{H}^\prime$ as before.
Let $\eta_1$ be the universal infinitesimal deformation with base $C_1$.
We  proceed by induction.  Suppose for some $k\geq 1$ we have constructed a finite dimensional local algebra $C_k$ and a  deformation $\eta_k$ of $L$ with base $C_k$.   
Let 
$$\mu:H_{Harr}^2(C_k;\mathbb{K})\longrightarrow (Ch_{2} (C_k))^\prime$$
be a homomorphism sending a cohomology class to a cocycle representing the class.  Let
$$f_{C_k}:Ch_{2} (C_k) \longrightarrow H_{Harr}^2(C_k;\mathbb{K})^\prime$$
 be the dual of $\mu$. Then we have the following extension of $C_k$:  
\begin{equation}\label{universal extension}
0\longrightarrow
 H_{Harr}^2(C_k;\mathbb{K})^\prime \stackrel{\bar {i}_{k+1}}{\longrightarrow} {\bar C}_{k+1}
\stackrel{\bar{p}_{k+1}}{\longrightarrow} C_k \longrightarrow 0.
\end{equation}
The corresponding {\it obstruction} $\theta_{\eta_{k}}([f_{C_k}]) \in H_{Harr}^2(C_k;\mathbb{K})^\prime \otimes HL^3(L;L)$ gives a linear map   
$\omega_k:H_{Harr}^2(C_k;\mathbb{K}) \longrightarrow HL^3(L;L)$
with the dual map  
$${\omega_k}^\prime:HL^3(L;L)^\prime \longrightarrow H_{Harr}^2(C_k;\mathbb{K})^\prime .$$
We have an induced extension 
$$ 0\longrightarrow coker (\omega'_{k})\longrightarrow \bar{C}_{k+1}/\bar{i}_{k+1}\circ \omega'_{k}(HL^3(L;L)')\longrightarrow C_k \longrightarrow 0.$$
Since $coker (\omega'_k)\cong (ker (\omega_k))^\prime$,
it yields an extension 
\begin{equation}\label{yields an extension}
0\longrightarrow (ker(\omega_k))^\prime \stackrel{i_{k+1}}{\longrightarrow} C_{k+1}
\stackrel{p_{k+1}}{\longrightarrow} C_k \longrightarrow 0
\end{equation}
where $C_{k+1}= { \bar{C}_{k+1}}/{\bar{i}_{k+1}\circ~ \omega_{k}^\prime (HL^3(L;L)')}$ and  $i_{k+1}$, $p_{k+1}$ are the mappings induced by $\bar{i}_{k+1}$ and $\bar{p}_{k+1}$, respectively. 
It turns out that the obstruction associated to the extension (\ref{yields an extension}) is $\omega|_{ker(\omega_k)}$.

As a consequence it is proved in \cite{FMM}
\begin{prop}
The deformation $\eta_k$ with base $C_{k}$ of a Leibniz algebra $L$ admits an extension to a deformation with base $C_{k+1}$, which is unique up to an isomorphism and an automorphism of the extension
$$0\longrightarrow
 (ker(\omega_k))^\prime \stackrel{i_{k+1}}{\longrightarrow} C_{k+1}
\stackrel{p_{k+1}}{\longrightarrow} C_k \longrightarrow 0.$$
\end{prop}

By induction, the above process yields a sequence of finite dimensional local algebras $C_{k}$ and deformations $\eta_{k}$ of the Leibniz algebra $L$ with base $C_{k}$ 
$$ \mathbb{K} \stackrel{p_{1}}{\longleftarrow} C_{1} \stackrel{p_{2}}{\longleftarrow} C_{2}\stackrel{p_{3}}{\longleftarrow} \ldots \ldots \stackrel{p_{k}}{\longleftarrow} C_{k}\stackrel{p_{k+1}}{\longleftarrow} C_{k+1}\ldots$$
such that $ {p_{k+1}}_{*} \eta_{k+1}=\eta_{k}$.
Thus by taking the projective limit we obtain a formal deformation $\eta$ of $L$ with base $C=\ilim_{k\rightarrow
\infty} C_{k}$.

\section{Massey Brackets and Obstructions}\label{Massey Brackets and Obstructions}

After constructing the universal infinitesimal deformation, one would like to extend it to higher order deformation. For this we need to compute obstructions. The standard procedure is to relate obstructions to Massey brackets. The connection between these two notions was first noticed in \cite{D}. A general approach to treat Massey brackets is given in \cite{FuL}. This approach is used to establish connection between Massey brackets and obstructions arising from Lie algebra deformations.

The aim of this section is to apply results in \cite{FuL} to relate Massey brackets to obstructions in the deformation of Leibniz algebras. A special case of the general definition is an inductive definition of Retakh (\cite{R, FuL}) which is useful for computational purposes.  

Suppose $(\mathcal{L},\nu,d)$ is a differential graded Lie algebra. We denote by $\mathcal{H}= \bigoplus_i \mathcal{H}^i$, the cohomology of $\mathcal{L}$ with respect to the differential $d$.
Let $F$ be a graded cocommutative coassociative coalgebra, that is a  graded vector space with a degree $0$ mapping (comultiplication) $\Delta:F\longrightarrow F\otimes F$ satisfying the conditions $S\circ \Delta=\Delta $ and $(1\otimes \Delta)\circ \Delta= (\Delta \otimes 1)\circ \Delta$, where 
$$S:F\otimes F\longrightarrow F\otimes F $$ is defined as $$S(\phi \otimes \psi)=(-1)^{|\phi||\psi|}(\psi \otimes \phi).$$
Suppose also that a filtration $F_0 \subset F_1\subset F$  is given in $F$, such that $F_0 \subset ker (\Delta)$ and $ Im (\Delta) \subset F_1\otimes F_1$. We need the following result (see \cite{FuL}.) 
\begin{prop}\label{horizontal}
Suppose a linear mapping $\alpha:F_1 \longrightarrow \mathcal{L}$ of degree $1$ satisfies the condition 
\begin{equation}\label{for Massey brackets}
d \alpha =\nu \circ (\alpha \otimes \alpha)\circ \Delta.
\end{equation}
Then $\nu \circ (\alpha \otimes \alpha)\circ \Delta (F)\subset ker (d) $.
\end{prop}

\begin{defn}\label{definition of Massey Bracket}
Let $a:F_0\longrightarrow \mathcal{H}$, $b:F/F_1\longrightarrow \mathcal{H}$ be two linear maps of degree $1$. We say that $b$ is contained in the Massey $F$-bracket of $a$, and write $b \in [a]_F$, or $b \in [a]$, if there exists a degree $1$ linear mapping $\alpha:F_1 \longrightarrow \mathcal{L}$ satisfying condition (\ref{for Massey brackets}) and such that the following diagrams are commutative, where the vertical maps labeled by $\pi$ denote the projections of each space onto the quotient space.

\begin{figure}[htb]
\begin{center}
\includegraphics[width=8.5cm]{masseyF1.epsi}
\end{center}
\caption{}
\end{figure}
\end{defn}

Note that the upper horizontal maps of the above diagrams are well defined, since $\alpha(F_0)\subset \alpha(ker \Delta)\subset ker (d)$ by virtue of (\ref{for Massey brackets}), and $\nu \circ (\alpha \otimes \alpha)\circ \Delta(F)\subset ker (d)$ by Proposition \ref{horizontal}.

The definition makes sense even if $F_1=F$. In that case $Hom (F/F_1,~\mathbb{K})=0$, and $[a]_F$ may either be empty or contain 0. In that case we say that $a ~ satisfies~ the~ condition ~of~ triviality~ of~ Massey ~F\mbox{-}brackets$.

Let $A$ be a complete local algebra with $1$ and augmentation $\varepsilon$. Let $\mathfrak{M}=ker(\varepsilon)$.
Let $\rho: (A \otimes L)\times (A \otimes L)\longrightarrow (A \otimes L)$ be a $A$-bilinear operation on $A \otimes L$ ($\rho$ need not satisfy the Leibniz identity) such that $\varepsilon \otimes id :A \otimes L \longrightarrow L$ is a homomorphism with respect to the operation $\rho$ on $A \otimes L$ and the usual bracket operation on $L$ in other words,
$$(\varepsilon \otimes id )\circ \rho(a_1 \otimes l_1, a_2 \otimes l_2)= \varepsilon (a_1 a_2) [l_1,l_2].$$
Note that for $1\otimes l_1,~1\otimes l_2 \in A \otimes L$ we have
$$(\varepsilon \otimes id)\circ \rho(1\otimes l_1,1\otimes l_2)= \varepsilon(1)[l_1,l_2]=\varepsilon \otimes id (1 \otimes [l_1,l_2])$$
Therefore \begin{equation}\label{rho}
 \rho(1\otimes l_1,1 \otimes l_2)-1 \otimes [l_1,l_2] \in ker (\varepsilon \otimes id )=ker (\varepsilon) \otimes L=\mathfrak{M}\otimes L.
 \end{equation}
 We consider the differential graded Lie algebra $(CL^{*}(L;L),\nu, d)$. Let $F=F_1=\mathfrak{M}'$, the dual of $\mathfrak{M}$ and $F_0=(\mathfrak{M}/\mathfrak{M}^2)^\prime$. Let $\Delta: F \longrightarrow F\otimes F$ be the comultiplication in $F$ which is the dual of the multiplication in $\mathfrak{M}$. Then $F$ is a cocommutative coassociative coalgebra.
For a linear functional $\phi:{\mathfrak{M}}\longrightarrow \mathbb{K}$ define a map $\alpha_{\phi}:L\otimes L \longrightarrow L$ by  
$$ \alpha_{\phi}(l_1,l_2)= (\phi \otimes id)(\rho(1\otimes l_1,1 \otimes l_2)-1 \otimes [l_1,l_2] ).$$
This gives $\alpha:{\mathfrak{M}}'\longrightarrow CL^2(L;L)$ by $\phi \mapsto \alpha_{\phi}$. From the definition it is clear that $\rho$ and $\alpha$ determine each other. 
Then we have 
\begin{prop}\label{Leibniz identity iff}
The operation $\rho$ satisfies the Leibniz identity if and only if $\alpha$ satisfies the   equation $d\alpha -\frac{1}{2}\nu\circ(\alpha\otimes \alpha)\circ \Delta=0$.
\end{prop}
\begin{proof}
 Let $\{m_i\}$ be a basis of $\mathfrak{M}$. Using (\ref{rho}) we can write $$\rho(1\otimes l_1,1 \otimes l_2)=1 \otimes [l_1,l_2]+\sum_{i}m_i\otimes \psi_i(l_1,l_2)$$ where $\psi_i \in CL^2(L;L)$ is given by $\psi_i=\alpha_{m^\prime_i}$. 

\begin{equation*}
 \begin{split}
 \mbox{Thus}~~&\rho(1\otimes l_1 ,\rho(1\otimes l_2,1\otimes l_3))\\
=&~\rho(1 \otimes l_1,1\otimes [l_2,l_3]+\sum_{i}m_i\otimes \psi_i(l_2,l_3))\\
=&~\rho(1\otimes l_1,1\otimes[l_2,l_3])+\sum_{i}m_i\rho(1\otimes l_1,1\otimes \psi_i(l_2,l_3))\\
=&~1\otimes [l_1,[l_2,l_3]]+\sum_{i}m_i \otimes \psi_i(l_1,[l_2,l_3])+\sum_{i}m_i \otimes [l_1,\psi_i(l_2,l_3)]\\
&~~~+\sum_{i,j}m_im_j \otimes \psi_j(l_1,\psi_i(l_2,l_3)).
 \end{split}
\end{equation*}

\begin{equation*}
 \begin{split}
\mbox{Similarly}~~ &\rho(\rho(1\otimes l_1,1\otimes l_2),1\otimes l_3)\\
=&~1\otimes [[l_1,l_2],l_3]
+\sum_{i}m_i \otimes \psi_i([l_1,l_2],l_3)+\sum_{i}m_i \otimes [\psi_i(l_1,l_2),l_3]\\
&~~~+\sum_{i,j}m_im_j \otimes \psi_j(\psi_i(l_1,l_2),l_3)
 \end{split}
\end{equation*}

\begin{equation*}
\begin{split}
\mbox{and}~~&\rho(\rho(1\otimes l_1,1\otimes l_3),1\otimes l_2)\\
=&~1\otimes [[l_1,l_3],l_2]
+\sum_{i}m_i \otimes \psi_i([l_1,l_3],l_2)+\sum_{i}m_i \otimes [\psi_i(l_1,l_3),l_2]\\
&~~~+\sum_{i,j}m_im_j \otimes \psi_j(\psi_i(l_1,l_3),l_2).
 \end{split}
\end{equation*}
For any linear functional $\phi:\mathfrak{M} \longrightarrow \mathbb{K}$, let $\phi(m_i)=x_i \in \mathbb{K}$. Then by (\ref{rho})
\begin{equation*}
\begin{split}
\alpha_{\phi}(l_1,l_2)=&(\phi \otimes id)(\sum_{i}m_i\otimes \psi_i(l_1,l_2))\\
=&\sum_{i}x_i \otimes \psi_i(l_1,l_2)\\
=&1\otimes (\sum_{i}x_i \psi_i)(l_1,l_2).
 \end{split}
\end{equation*}
So, $\alpha_{\phi}$ can be expressed as $\sum_{i}x_i \psi_i$.
 Let $\Delta(\phi)=\sum_{p}\xi_p \otimes \eta_p~~~\mbox{for some}~\xi_p,\eta_p \in \mathfrak{M}'$. We set $\xi_p(m_i)=\xi_{p,i}~~\mbox{and}~\eta_p(m_i)=\eta_{p,i}$. Thus 
 \begin{equation*}
 \begin{split}
 \phi(m_i~ m_j)=&~\Delta(\phi)(m_i\otimes m_j)
 =(\sum_{p}\xi_p\otimes \eta_p)(m_i\otimes m_j)
 =\sum_{p}\xi_{p,i}~\eta_{p,j}.
 \end{split}
 \end{equation*}

\begin{equation*}
 \begin{split}
\mbox{Now}~~~&(\phi \otimes id)( \sum_{i,j}m_im_j \otimes \psi_j(l_1,\psi_i(l_2,l_3))\\
 =&~\sum_{i,j,p}\xi_{p,i}~\eta_{p,j}~\psi_j(l_1,\psi_i(l_2,l_3))\\
 =&~\sum_{p}(\sum_{i}\xi_{p,i}(\sum_{j}\eta_{p,j}\psi_j(l_1,\psi_i(l_2,l_3))))\\
 =&~\sum_{p}(\sum_{i}\xi_{p,i}\alpha_{\eta_p}(l_1,\psi_i(l_2,l_3)))\\
 =&~\sum_{p}\alpha_{\eta_p}(l_1,\sum_i\xi_{p,i}\psi_i(l_2,l_3))\\
 =&~\sum_{p}\alpha_{\eta_p}(l_1,\alpha_{\xi_p}(l_2,l_3)).
\end{split}
\end{equation*}
\begin{equation*}
 \begin{split}
\mbox{Therefore}~~&(\phi \otimes id)( \rho(1\otimes l_1 ,\rho(1\otimes l_2,1\otimes l_3)))\\
 =&~ \sum_{i}\phi(m_i)\otimes \psi_i(l_1,[l_2,l_3])+\sum_{i}\phi(m_i)\otimes [l_1,\psi_i(l_2,l_3)] \\
 &~~~+\sum_{p}\alpha_{\eta_p}(l_1,\alpha_{\xi_p}(l_2,l_3))\\ =&~\alpha_{\phi}(l_1,[l_2,l_3])+[l_1,\alpha_{\phi}(l_2,l_3)]
 +\sum_{p}\alpha_{\eta_p}(l_1,\alpha_{\xi_p}(l_2,l_3)).
 \end{split}
\end{equation*}
\begin{equation*}
 \begin{split}
\mbox{Similarly}~~ &(\phi \otimes id)(\rho(\rho(1\otimes l_1,1\otimes l_2),1\otimes l_3))\\
=&~\alpha_{\phi}([l_1,l_2],l_3)+[\alpha_{\phi}(l_1,l_2),l_3]
 +\sum_{p}\alpha_{\eta_p}(\alpha_{\xi_p}(l_1,l_2),l_3).
 \end{split}
\end{equation*}

\begin{equation*}
 \begin{split}
\mbox{and}~~&(\phi \otimes id)(\rho(\rho(1\otimes l_1,1\otimes l_3),1\otimes l_2))\\
=&~\alpha_{\phi}([l_1,l_3],l_2)+[\alpha_{\phi}(l_1,l_3),l_2]
 +\sum_{p}\alpha_{\eta_p}(\alpha_{\xi_p}(l_1,l_3),l_2).
 \end{split}
\end{equation*}

 \begin{equation*}
 \begin{split}
\mbox{Hence we get,}~&(\phi \otimes id)(\rho(1\otimes l_1 ,\rho(1\otimes l_2,1\otimes l_3))-\rho(\rho(1\otimes l_1, 1\otimes l_2),1\otimes l_3 )\\
 &~~~+\rho(\rho(1\otimes l_1, 1\otimes l_3),1\otimes l_2 ))\\
 =&~\delta \alpha_{\phi}(l_1,l_2,l_3)+\frac{1}{2}\sum_{p}[\alpha_{\eta_p},\alpha_{\xi_p}](l_1,l_2,l_3)\\
 =&~(-d\alpha +\frac{1}{2}\nu\circ(\alpha\otimes \alpha)\circ \Delta )\phi(l_1,l_2,l_3). 
 \end{split}
\end{equation*}
Thus it follows that $\rho$ satisfies the Leibniz identity if and only if $\alpha$ satisfies the equation $d\alpha -\frac{1}{2}\nu \circ(\alpha\otimes \alpha)\circ \Delta=0$.
\end{proof}

It follows from  Proposition \ref{Leibniz identity iff} that for a deformation $\rho$ of $L$, $\alpha(F_0)\subset ker(d)$ as $F_0 \subset ker(\Delta)$. 
Let $a$ denote the composition 
$$a:F_0 \stackrel{\alpha}\longrightarrow ker(d) \stackrel{\pi}\longrightarrow \mathbb{H}~~\mbox{where}~\mathbb{H}=HL^2(L;L).$$
Then the following is a consequence of Proposition \ref{Leibniz identity iff} and definition of Massey $F$ bracket. 
\begin{cor}
A linear map $a:F_0 \longrightarrow \mathbb{H}$ is a differential of some deformation with base $A$ if and only if $\frac{1}{2}a$ satisfies the condition of triviality of Massey $F$-brackets.
\end{cor}

Next we relate the obstruction  $\omega_k$ at the $k$th stage in the construction of versal deformation to Massey brackets.
Consider the sequence of finite dimensional local  algebras $C_{k}$ with maximal ideals $\mathfrak{M}_k$ and deformations $\eta_{k}$ of the Leibniz algebra $L$ with base $C_{k}$ yielding an inverse system
$$ \mathbb{K} \stackrel{p_{1}}{\longleftarrow} C_{1} \stackrel{p_{2}}{\longleftarrow} C_{2}\stackrel{p_{3}}{\longleftarrow} \ldots \ldots \stackrel{p_{k}}{\longleftarrow} C_{k}\stackrel{p_{k+1}}{\longleftarrow} C_{k+1}\ldots$$ 
$$\mbox{where}~~ {p_{k+1}}_{*} \eta_{k+1}=\eta_{k}.$$
Taking the dual we get the direct system
$$ \mathbb{K} \stackrel{p' _{1}}{\longrightarrow} C'_{1} \stackrel{p' _{2}}{\longrightarrow} C'_{2}\stackrel{p' _{3}}{\longrightarrow} \ldots \ldots \stackrel{p' _{k}}{\longrightarrow} C'_{k}\stackrel{p'_{k+1}}{\longrightarrow} C'_{k+1}\ldots.$$
Also, by considering the maximal ideals $\mathfrak{M}_k$ we get another system
$$ \mathbb{K} \stackrel{p' _{1}}{\longrightarrow} \mathfrak{M}'_{1} \stackrel{p' _{2}}{\longrightarrow} \mathfrak{M}'_{2}\stackrel{p' _{3}}{\longrightarrow} \ldots \ldots \stackrel{p' _{k}}{\longrightarrow} \mathfrak{M}'_{k}\stackrel{p'_{k+1}}{\longrightarrow} \mathfrak{M}'_{k+1}\ldots$$
where each $p'_k$ is injective.
In the induction process we get an extension of $C_k$ given by 
$$0\longrightarrow
 H_{Harr}^2(C_k;\mathbb{K})^\prime \stackrel{\bar {i}_{k+1}}{\longrightarrow} {\bar C}_{k+1}
\stackrel{\bar{p}_{k+1}}{\longrightarrow} C_k \longrightarrow 0$$
where the obstruction for extending $\eta_k$ to a deformation of $L$ with base $\bar{C}_{k+1}$ is given by 
$\omega_k:H_{Harr}^2(C_k;\mathbb{K}) \longrightarrow HL^3(L;L).$
To make this obstruction zero we consider $$C_{k+1}= { \bar{C}_{k+1}}/{\bar{i}_{k+1}\circ~ \omega_{k}^\prime (HL^3(L;L))} .$$
Let $F=(\bar {\mathfrak{M}}_{k+1})';~F_1=\mathfrak{M}'_k~~\mbox{and}~F_0=\mathfrak{M}'_{1}=\mathbb{H}$.\\
Thus $F/F_1=H^2_{Harr}(C_k;\mathbb{K})$ and $\omega_k$ can be viewed as a map
$$\omega_k:F/F_1\longrightarrow HL^3(L;L).$$
\begin{thm}\label{obs at each stage}
The obstruction $\omega_k$ has the property, 
$2\omega_k \in [id]_F$. Moreover, an arbitrary element of $[id]_F$ is equal to $2\omega_k$ for an appropriate extension of the deformation $\eta_1$ of $L$ with base $C_1$ to a deformation $\eta_k$ of $L$ with base $C_k$.
\end{thm}
\begin{proof}
As before we define a map
$$ \alpha: \mathfrak{M}'_k\longrightarrow CL^2(L;L)$$
by $\alpha_{\phi}(l_1,l_2)=(\phi\otimes id)([1\otimes l_1,1\otimes l_2]_{\eta_k}-1\otimes [l_1,l_2])
~~\mbox{for}~\phi \in \mathfrak{M}'_k~~\mbox{and}~l_1,l_2 \in L$, using the deformation $\eta_k$ with base $C_k$. 
Since $\eta_k$ is a Leibniz algebra structure on $C_k\otimes L$,  Proposition \ref{Leibniz identity iff} implies $d\alpha=\frac{1}{2}\nu \circ(\alpha\otimes \alpha)\circ\Delta$. It is clear that different $\alpha$ with these properties corresponds to different extensions $\eta_k$ of $\eta_1$.

Observe that $\alpha|_{F_0}:F_0\longrightarrow CL^2(L;L)$ is given by $\alpha|_{F_0}(h_i)=\mu(h_i)$, a representative of the cohomology class $h_i$. So
$$\alpha|_{F_0}:F_0\longrightarrow CL^2(L;L)\stackrel{\pi}\longrightarrow \mathbb{H}$$ 
gives  $a:F_0\longrightarrow \mathbb{H}$, the identity map.\\
In the definition of Massey $F$-bracket, the map $b:F/F_1\longrightarrow HL^3(L;L)$ is represented by the map $\nu \circ(\alpha\otimes \alpha)\circ \Delta:F\longrightarrow CL^3(L;L)$. In our case the obstruction is given by  $\omega_k:H_{Harr}^2(C_k;\mathbb{K}) \longrightarrow HL^3(L;L)$.
Consider  a basis $\{m_i\}_{1\leq i \leq r}$ of $\mathfrak{M}_k$  and extend it to a basis $\{\bar{m}_i \}_{1\leq i \leq r+s}$ of $\bar{\mathfrak{M}}_{k+1}$. Now we can write $$[1\otimes l_1,1\otimes l_2]_{\eta_k}=1\otimes [l_1,l_2]+\sum_{i=1}^r m_i \otimes \psi_i(l_1,l_2) .$$ 
Then by definition of $\alpha$ we have $\alpha(m'_i)(l_1,l_2)=\psi_i(l_1,l_2)~~\mbox{for}~i\geq r$.\\
For arbitrary cochains $\psi_i \in CL^2(L;L)~~\mbox{for}~r+1 \leq i\leq s$ the $\bar{C}_{k+1}$-bilinear map $\{,\}$ on $\bar{C}_{k+1}\otimes L$  is given by
 $$\{1\otimes l_1,1\otimes l_2\}=1\otimes[l_1,l_2]+\sum_{i=1}^{r+s} \bar{m}_i \otimes \psi_i(l_1,l_2).$$  
Let the multiplication in $\bar{\mathfrak{M}}_{k+1}$ be defined (on the basis) as $$\bar{m}_i~\bar{m}_j=\sum_{p=1}^{r+s}c_{i~j}^p \bar{m}_p.$$
Then $\Delta:(\bar{\mathfrak{M}}_{k+1})^\prime \longrightarrow \mathfrak{M}_{k}^\prime \otimes \mathfrak{M}_k ^\prime$ is given by $\Delta(\bar{m}'_{p})=\sum_{i,j=1}^{s}c_{ij}^p m'_i \otimes m'_j$.
Now
\begin{equation*}
\begin{split}
&\{\{1\otimes l_1,1\otimes l_2 \},1\otimes l_3\}\\
=&~\{1\otimes [l_1,l_2]+\sum_{i=1}^{r+s}\bar{m}_i \otimes \psi_i(l_1,l_2),1\otimes l_3\}\\
=&~1\otimes [[l_1,l_2],l_3]+\sum_{i=1}^{r+s}\bar{m}_i\otimes \psi_i ([l_1,l_2],l_3)
+\sum_{i=1}^{r+s}\bar{m}_i\otimes [\psi_i(l_1,l_2),l_3]\\
&~~~+\sum_{i,j=1}^{r}\bar{m}_j \bar{m}_i \otimes \psi_j(\psi_i(l_1,l_2),l_3)\\
=&~1\otimes [[l_1,l_2],l_3]+\sum_{i=1}^{r+s}\bar{m}_i\otimes \psi_i ([l_1,l_2],l_3)
+\sum_{i=1}^{r+s}\bar{m}_i\otimes [\psi_i(l_1,l_2),l_3]\\
&~~~+\sum_{i,j=1}^{r}\sum_{p=1}^{r+s}c_{ij}^p \bar{m}_p \otimes \psi_j(\psi_i(l_1,l_2),l_3).
\end{split}
\end{equation*}

\begin{equation*}
\begin{split}
\mbox{Similarly}~~&\{\{1\otimes l_1,1\otimes l_3 \},1\otimes l_2\}\\
=&~1\otimes [[l_1,l_3],l_2]+\sum_{i=1}^{r+s}\bar{m}_i\otimes \psi_i ([l_1,l_3],l_2)
+\sum_{i=1}^{r+s}\bar{m}_i\otimes [\psi_i(l_1,l_3),l_2]\\
&~~~+\sum_{i,j=1}^{r}\sum_{p=1}^{r+s}c_{ij}^p \bar{m}_p \otimes \psi_j(\psi_i(l_1,l_3),l_2)
\end{split}
\end{equation*}

\begin{equation*}
\begin{split}
\mbox{and}~~&\{1\otimes l_1,\{1\otimes l_2,1\otimes l_3\}\}\\
=&~1\otimes [l_1,[l_2,l_3]]+\sum_{i=1}^{r+s}\bar{m}_i\otimes \psi_i(l_1,[l_2,l_3])
+\sum_{i=1}^{r+s}\bar{m}_i\otimes [l_1,\psi_i(l_2,l_3)]\\
&~~~+ \sum_{i,j=1}^{r}\sum_{p=1}^{r+s}c_{ij}^p \bar{m}_p \otimes \psi_j(l_1,\psi_i(l_2,l_3)).
\end{split}
\end{equation*}

\begin{equation*}
\begin{split}
\mbox{Therefore}~~&(\bar{m}'_p \otimes id)(\{1\otimes l_1,\{1\otimes l_2,1\otimes l_3\}\}-\{\{1\otimes l_1,1\otimes l_2 \},1\otimes l_3\}\\
&~~~+\{\{1\otimes l_1,1\otimes l_3 \},1\otimes l_2\} )\\
=&~\delta \psi_p (l_1,l_2,l_3)+\frac{1}{2}\sum_{i,j=1}^{r}c_{ij}^p [\psi_j,\psi_i](l_1,l_2,l_3)\\
=&~\delta \psi_p (l_1,l_2,l_3)+\frac{1}{2}\nu\circ(\alpha \otimes \alpha)\circ \Delta (\bar{m}'_{p})(l_1,l_2,l_3).
\end{split}
\end{equation*}
Taking $b=2\omega_k$ and $a=id|_\mathbb{H}$ in Definition \ref{definition of Massey Bracket} the result follows.
\end{proof}

\section{Computations for the Leibniz algebra $\lambda_6$}\label{computation}
To construct a versal deformation of $\lambda_6$, we need to compute the second and third cohomology space of  $\lambda_6=L$. First consider $HL^2(L;L)$.
Our computation consists of the following steps:\\
 (i) To determine a basis of the space of cocycles $ZL^2(L;L)$,\\
 (ii) to find out a basis of the coboundary space $BL^2(L;L)$,\\
 (iii) to determine the quotient space  $HL^2(L;L)$.\\
(i) Let $\psi$ $\in$ $ZL^2(L;L)$. Then $\psi :L\otimes L\longrightarrow L$ is a linear map and  $\delta \psi =0$, where  
 \begin{equation*}
 \begin{split}
 \delta \psi(e_i, e_j, e_k)
 &=[e_i,\psi(e_j, e_k)]+[\psi (e_i, e_k), e_j]-[\psi(e_i, e_j), e_k] -\psi([e_i, e_j], e_k) \\
&~ +\psi(e_i,[e_j,e_k])+\psi([e_i, e_k], e_j) ~\mbox{for}~0\leq i,j,k \leq 3.
\end{split}   
\end{equation*}
Suppose  $\psi(e_i,e_j)=\sum_{k=1} ^{3} a_{i,j}^{k} e_k$ where $a_{i,j}^{k} \in \mathbb C$
  ; for $1\leq i,j,k\leq3$.
Since $\delta \psi =0$ equating the coefficients  of $e_1, e_2
~\mbox{and}~ e_3 $ in $\delta \psi(e_i, e_j, e_k)$ we get the following relations:
\begin{equation*}
\begin{split}
&(i)~ a_{1,1}^1 =a_{1,1}^3=0 ;\\
&(ii)~ a_{1,2}^1=a_{1,2}^3=0 ;\\
&(iii)~ a_{2,1}^1 =a_{2,1}^2=a_{2,1}^3=0 ;\\
&(iv)~a_{2,2}^1=a_{2,2}^2=a_{2,2}^3=0 ;\\
&(v)~a_{3,1}^2=a_{3,1}^3=0 ;\\
&(vi)~ a_{3,2}^2=a_{3,2}^3=0;\\
&(vii)~ a_{2,3}^3=0 ;\\
&(viii)~a_{1,1}^2=a_{3,1}^1=-a_{3,3}^3;\\
&(ix)~a_{1,2}^2=-a_{1,3}^3=a_{3,2}^1 .
\end{split}
\end{equation*}
Observe that there
is no relation among $a_{1,3}^1$,$a_{1,3}^2$, $a_{2,3}^1$, $a_{2,3}^2$, $a_{3,3}^1 $  and
$a_{3,3}^2$. Therefore, in terms of the ordered basis $\{e_1\otimes e_1, e_1\otimes e_2, e_1\otimes e_3, e_2\otimes e_1, e_2\otimes e_2, e_2\otimes e_3, e_3\otimes e_1, e_3\otimes e_2, e_3\otimes e_3\}$ of $L\otimes L$ and $\{e_1, e_2, e_3\}$ of $L$, the matrix corresponding to $\psi $ is of the form
$$M= \left( \begin{array}{llrllllll}
0&   0  &  x_3  &0 &0 &x_5 &x_1  &x_2 &x_7 \\
x_1& x_2&  x_4  &0 &0 &x_6 &0    &0   &x_8    \\
0&   0  & -x_2  &0 &0 &  0 &0    &0   &-x_1
\end{array}  \right)$$
where $$x_1=a_{1,1}^2 ; x_2=a_{1,2}^2 ; x_3=a_{1,3}^1; x_4=a_{1,3}^2;
x_5=a_{2,3}^1; x_6=a_{2,3}^2; x_7=a_{3,3}^1; x_8=a_{3,3}^2$$ are in $\mathbb C$~.
Let $\phi_i \in ZL^2(L;L)$ for $1 \leq i\leq 8$, be the cocycle with 
$x_i=1$ and $x_j=0$ for $i\neq j$ in the above matrix of $\psi$. It is easy to check that $\{\phi_1,\cdots, \phi_8\}$ forms a basis of $ZL^2(L;L)$.

(ii) 
Let $ \psi_0 \in BL^2(L;L)$. We have $\psi_0=\delta g$ for some $1$-cochain $g \in CL^1(L;L)=Hom(L;L)$. Suppose the matrix associated to $\psi_0$ is same as the above matrix $M$.

Let  $g(e_i)=g_i ^1 e_1 +g_i ^2 e_2+g_i ^3 e_3$ for $i=1,2,3$. 
The matrix associated to $g$ is given by 
\begin{center} $\left(\begin{array}{lll}
g_1 ^1& g_2 ^1  & g_3 ^1\\
g_1 ^2& g_2 ^2  & g_3 ^2\\
g_1 ^3& g_2 ^3  & g_3 ^3
\end{array} \right).$ \end{center}
 From the definition of coboundary we get $$\delta g(e_i,e_j)=[e_i,g(e_j)]+[g (e_i),e_j]-g([e_i,e_j])$$ for $0\leq i,j \leq 3$. The matrix $\delta g$ can be written as 
 
 \begin{center}$
  \left( \begin{array}{rrrrrrrll}
0        & 0       & (g_{1}^3-g_{2}^1)             &0 &0 &   g_{2}^3  &g_{1}^3&g_{2}^3 &(2g_{3}^3-g_{1}^1) \\
g_{1}^3  & g_{2}^3 &  (g_{3}^3 +g_{1}^1-g_{2}^2)   &0 &0 &   g_{2}^1  &0      &0       &(g_{3}^1-g_{1}^2)    \\
0        & 0       &   -g_{2}^3                     &0 &0 &    0       &0      &0       &-g_{1}^3
\end{array}  \right).$ \end{center}

Since $\psi_0=\delta g$ is also a cocycle in $CL^2(L;L)$, comparing matrices $\delta g$ and $M$ we conclude that the matrix of $\psi_0$  is of the form 
\begin{center}$
  \left( \begin{array}{rrrrrrrll}
0    & 0   &  x_3 &0 &0 &x_2       &x_1  &x_2&x_7 \\
x_1  & x_2 & x_4  &0 &0 &(x_1-x_3) &0  &0  &x_8    \\
0    & 0   & -x_2 &0 &0 &  0       &0 &0  &-x_1
\end{array}  \right).$ \end{center}
Let ${\phi_i}^\prime \in BL^2(L;L)~\mbox{for}~i=1,2,3,4, 7,8$ be the coboundary with $x_i=1$ and $x_j=0$ for $i\neq j$ in the above matrix of
$\psi_0$. It follows that $\{\phi_1^\prime,\phi_2^\prime,\phi_3^\prime,\phi_4^\prime,\phi_7^\prime,\phi_8^\prime\}$ forms a basis of the coboundary space $BL^2(L;L)$.

(iii)
It is  straightforward to check that  $[\phi_2]$ and $[\phi_3]$ span $HL^2(L;L)$ where $[\phi_i]$ denotes the cohomology class represented by the cocycle $\phi_i$. 

Thus  $dim(HL^2(L;L))=2$.

Next let us consider $HL^3(L;L)$.
If $\psi \in ZL^3(L;L)$, then a computation similar to $2$-cocycles shows that the transpose of the matrix of $\psi$ is 
\[ \left( \begin{array}{rrr}
0                 &x_1                           &   0                              \\
0                 &x_2                           &  0 \\
x_3               &x_4                           & (x_2+x_5)                     \\
0                 &x_5                           &  0                            \\
0                 &0                             &0                              \\
x_6               &x_{17}                       &0                               \\
x_7               &x_8                          & -x_5                            \\
\frac{1}{5}(2x_2-3x_{6}+2x_{11})      &(x_{13} - x_{10}+2x_7+x_3-2x_{1})   & 0     \\
(2x_{16}-x_{14})  &x_9                           &x_1                              \\
0                 &0                            &0                                \\
0                 &0                            &0                                 \\
\frac{1}{5}(3x_{2}+3x_6-2x_{11})-x_5      &x_{10}                       &0       \\
0                &0                              &0 \\
0                &0                              &0 \\
0                 &x_{11}                           &0 \\
x_5              & (x_1-x_7)                      &0                             \\
0                &\frac{1}{5}(3x_{2}+3x_6-2x_{11})                    &0       \\
(x_1-x_7)         &(3x_{16}-x_{14}-x_{8})         &x_5 \\
x_1                 &0                              &0 \\
x_2                &0                              &0 \\
x_{12}              &x_{18}                          &x_{13}\\
x_5                  &0                              &0 \\
0                     &0                            &0 \\
(x_{17}-x_{13}-x_{10}+3x_7+2x_3)               &x_{19}                      &\frac{1}{5}(6x_{2}+x_{6}+x_{11}) \\
x_{14}                &x_{15}                            &-x_1 \\
(2x_{13}-2x_{1}-x_3-x_7) & (x_{14}+x_{12}-x_8-x_4)   &-x_2 \\
(x_9+x_{15})           &x_{20}                    &x_{16}\end{array} \right).\]

Let ${\tau_i} \in ZL^3(L;L)$ for $1\leq i \leq 20$ be the cocycle with  $x_i=1$ and $x_j=0$ for $i\neq j$ in the
above matrix.
Then  one can check that $\{\tau_i\}_{1\leq i\leq 20}$ forms a basis
of $ZL^3(L;L)$. So $dim(ZL^3(L;L))=20$.

On the other hand suppose $\psi \in CL^3(L;L)$ is a coboundary with $\psi=\delta g$. Let $g(e_i,e_j)=g_{i,j} ^1 e_1 +g_{i,j} ^2 e_2+g_{i,j} ^3 e_3$;  for $1 \leq i,j \leq 3$. Then the transpose of the matrix of $\psi=\delta g$ is 

$$ \left(\begin{array}{ccc}
  0                                 & g_{1,1}^{3}                                     & 0 \\
0                                   & g_{1,2}^{3}                                     & 0 \\
(g_{2,1}^{1}+g_{1,2}^{1}-g_{1,1}^{3}) &( g_{2,1}^{2}+g_{1,2}^{2}-g_{1,1}^{1}+g_{1,3}^{3}) &( g_{2,1}^{3}+g_{1,2}^{3}) \\
  0 & g_{2,1}^{3} & 0 \\
  0 & g_{2,2}^{3} & 0 \\
  (g_{2,2}^{1}-g_{1,2}^{3}) &(g_{2,2}^{2}+g_{2,3}^{3}-g_{1,2}^{1})  &g_{2,2}^{3}  \\
   (g_{1,1}^{3}-g_{2,1}^{1})& (g_{1,1}^{1}+g_{3,1}^{3}-g_{2,1}^{2}) &-g_{2,1}^{3}  \\
    (g_{1,2}^{3}-g_{2,2}^{1}) & (g_{1,2}^{1}+g_{3,2}^{3}-g_{2,2}^{2}) & -g_{2,2}^{3} \\
     g_{1,1}^{1} & (g_{3,3}^{3}+g_{1,1}^{2}) & g_{1,1}^{3} \\
     0 & 0 & 0 \\
     0 & 0 & 0 \\
    (g_{2,2}^{1}-g_{2,1}^{3})& (g_{2,2}^{2}-g_{2,1}^{1}) & g_{2,2}^{3} \\
     0 & 0 & 0 \\
     0 & 0 & 0 \\
    -g_{2,2}^{3}&-g_{2,2}^{1}  & 0 \\
    g_{2,1}^{3}  &g_{2,1}^{1}  & 0 \\
    g_{2,2}^{3} & g_{2,2}^{1} & 0 \\
    g_{2,1}^{1} &g_{2,1}^{2}  & g_{2,1}^{3} \\
    g_{1,1}^{3}& 0 & 0 \\
    g_{1,2}^{3}  & 0 & 0 \\
   (g_{1,1}^{1}+g_{3,2}^{1}-g_{3,1}^{3}+g_{1,3}^{3}) & (g_{1,1}^{2}+g_{3,2}^{2}-g_{3,1}^{1}) & (g_{1,1}^{3}+g_{3,2}^{3}) \\
    g_{2,1}^{3}& 0 & 0 \\
    g_{2,2}^{3}& 0 & 0 \\
   (g_{2,3}^{3}-g_{3,2}^{3}+g_{1,2}^{1})&(g_{1,2}^{2}-g_{3,2}^{1})  & g_{1,2}^{3} \\
   (2g_{3,1}^{3}-g_{1,1}^{1}) & (g_{3,1}^{1}-g_{1,1}^{2}) & -g_{1,1}^{3} \\
   (2g_{3,2}^{3}-g_{1,2}^{1}) & (g_{3,2}^{1}-g_{1,2}^{2}) &  -g_{1,2}^{3} \\
   (g_{3,1}^{1}+g_{3,3}^{3})  & g_{3,1}^{2} &g_{3,1}^{3}  \\
                                 \end{array}
                               \right)$$
Since $\delta \psi$ is also zero, the transpose of the matrix of $\psi$ is of the previous form as well. Thus a coboundary $\psi$ has the following transpose matrix.

\[ \left( \begin{array}{rrr}
0                 &x_1                           &   0                              \\
0                 &x_2                           &  0 \\
x_3               &x_4                           & (x_2+x_5)                     \\
0                 &x_5                           &  0                            \\
0                 &0                             &0                              \\
-(x_2+x_{11})              &x_{17}                       &0                               \\
x_7               &x_8                          & -x_5                            \\
(x_{2}+x_{11})      &(x_{13} - x_{10}+2x_7+x_3-2x_{1})   & 0     \\
(2x_{16}-x_{14})  &x_9                           &x_1                              \\
0                 &0                            &0                                \\
0                 &0                            &0                                 \\
-(x_{11}+x_{5})     &x_{10}                       &0                                   \\
0                &0                              &0 \\
0                &0                              &0 \\
0                 &x_{11}                           &0 \\
x_5              & (x_1-x_7)                      &0                  \\
0                &-x_{11}                          &0                   \\
(x_1-x_7)         &(3x_{16}-x_{14}-x_{8})         &x_5 \\
x_1                 &0                              &0 \\
x_2                &0                              &0 \\
x_{12}              &x_{18}                          &x_{13}\\
x_5                  &0                              &0 \\
0                     &0                            &0 \\
(x_{17}-x_{10}+3x_7+2x_3-x_{13})               &(x_4 +x_8-x_{12}-x_{14})                      &x_2 \\
x_{14}                &x_{15}                            &-x_1 \\
(2x_{13}-2x_{1}-x_3-x_7) & (x_{14}+x_{12}-x_8-x_4)   &-x_2 \\
(x_9+x_{15})           &x_{20}                    &x_{16}\end{array} \right).\]

 This implies that $dim(BL^3(L;L))=18$. Consequently $dim(HL^3(L;L))=2$.

Since $HL^3(L;L)$ is nontrivial, it is necessary to compute possible obstructions in order to extend an infinitesimal deformation to a higher order one. 
 
First we describe the universal infinitesimal deformation for our Leibniz algebra.
To make our computation simpler, we choose the representative cocycles $\mu_1,\mu_2$ where $\mu_1=\phi_2 - \phi^\prime_2$ and $\mu_2=\phi_3$. Let us denote a dual basis in $HL^2(L;L)^\prime$ by $\{t,s\}$. By Remark \ref{exp of inf} the universal infinitesimal deformation of $L$ can be written as 
$$[1\otimes e_i,1\otimes e_j]_{\eta_1}=1\otimes [e_i,e_j]+ t \otimes \mu_1(e_i,e_j)+s \otimes \mu_2(e_i,e_j).$$  
 with base $C_1 =\mathbb{C}~\oplus \mathbb{C}~t ~ \oplus ~ \mathbb{C}~s$.

Let us describe a simpler version of the inductive definition of Massey brackets by Retakh  \cite{R}(see \cite{F}), relevant for Leibniz algebra deformations. These $n$ th order operations are partially defined and they are well defined modulo the $(n-1)$ th order ones. The second order operation is the superbracket in the cochain complex. More precisely, if $y_1=[x_1], y_2=[x_2]$ are $2$- cohomology classes, then the second order operation $<y_1,y_2>$ is represented by the superbracket $[x_1,x_2]$. 

Suppose that $y_i \in HL^2(L;L)$, $1\leq i\leq 3$ such that $<y_i,y_j>=0$ for every $i$ and $j$. This means that for a  cocycle $x_i$ representing $y_i$ we have $[x_i,x_j]=d x_{ij}$ for some $2$- cochain $x_{ij}$. Then the third order Massey operation $<y_1,y_2,y_3>$ is defined and is represented by 
$$ [x_{12},x_3]+[x_1,x_{23}]+[x_{13},x_2].$$
The cohomology class is independent of the choice of $x_{ij}$. The higher order Massey operations are defined inductively.  

Now  we compute the Massey brackets using the above definition.
\begin{description}
\item[(i)]By definition  $<[\mu_1],[\mu_1]>$ is represented by $[\mu_1,\mu_1]=2 (\mu_1 \circ \mu_1).$\\
Now $(\mu_1 \circ \mu_1)(e_i,e_j,e_k)\\
=\mu_1(\mu_1(e_i,e_j),e_k)-\mu_1(\mu_1(e_i,e_k),e_j)-\mu_1(e_i,\mu_1(e_j,e_k))$ for $1\leq i,j,k\leq 3$.

Since $\mu_1(e_2,e_3)=-e_1$ and takes value zero on all other basis element of $L\otimes L$, it follows that $\mu_1 \circ \mu_1=0$. 
\item[(ii)] Similarly $<[\mu_1],[\mu_2]>$ is represented by $[\mu_1,\mu_2]= \mu_1\circ \mu_2 +\mu_2\circ \mu_1$.
Since $\mu_2(e_1,e_3)=e_1$ and takes value zero on all other basis element of $L\otimes L$ it follows that $<[\mu_1],[\mu_2]>=0$. 
\item[(iii)] The bracket $<[\mu_2],[\mu_2]>$ is represented by $[\mu_2,\mu_2]=2(\mu_2 \circ \mu_2)=0$. 

\end{description}
Since $\{[\mu_1],[\mu_2]\}$ form a basis for $HL^2(L;L)$, it follows that all the Massey $2$- brackets are trivial. So all the Massey $3$- brackets are defined.

From the definition of Massey $3$- bracket it follows that all the Massey $3$- brackets  $<[\mu_i],[\mu_j],[\mu_k]>$ are trivial and represented by the $0$-cocycle. By induction it follows that any $<[\mu_1],[\mu_2],\cdots,[\mu_k]>=0$ for $[\mu_i]\in HL^2(L;L)$ and moreover, they are represented by the $0$-cocycle.

By  Theorem \ref{obs at each stage} and considering the inductive definition of Massey brackets in \cite{FuL} it follows that the possible obstruction at each stage in extending $\eta_1$ to a versal deformation with base $\mathbb{C}[[t,s]]$ can be realised as the Massey brackets of $\mu_1$ and $\mu_2$. So the possible obstruction vanishes.

As there are no obstructions to extending the universal infinitesimal deformation $\eta_1$,  it means that $\eta_1$ extends to a versal deformation with base $\mathbb{C}[[t,s]]$. Moreover, observe that by our choice of $\mu_1$ and $\mu_2$ every Massey brackets is represented by the $0$- cochain, and so $\eta_1$ is itself a Leibniz bracket with base $\mathbb{C}[[t,s]]$. It follows by the construction in \cite{FMM} that $\eta_1$ is a versal deformation.

Let us write out the versal deformation we have constructed:
\begin{equation*}
\begin{split}
& [e_1,e_3]_{t,s}=e_2+e_1s,~~ [e_3,e_3]_{t,s}=e_1,~~ [e_2,e_3]_{t,s}=-e_1 t\\
&\mbox{with all the other brackets of basis elements being 0}.
\end{split}
\end{equation*}

Thus we obtain the following two nonequivalent $1$-parameter deformations for the Leibniz algebra $\lambda_6$.
\begin{equation*}
\begin{split}
&(i)~[e_1,e_3]_t=e_2,~~ [e_2,e_3]_t=-e_1 t,~~ [e_3,e_3]_t=e_1\\
&\mbox{ all the other brackets of basis elements are zero,}~
\\
&(ii)~ [e_1,e_3]_s=e_2+e_1 s,~~ [e_3,e_3]_s=e_1\\
&\mbox{ all the other brackets of basis elements are zero.}
\end{split}
\end{equation*}

{\bf \large Conclusions:}
In this paper we computed a versal deformation of a $3$- dimensional nilpotent Leibniz algebra. For computing obstructions we introduced the notion of Massey brackets and proved the relationship between Massey brackets and obstructions. It turned out that in our example there are no obstructions in extending an infinitesimal deformation to a formal base, and so the universal infinitesimal deformation itself is versal with base $\mathbb{C}[[t,s]]$. From the computation it follows that our Leibniz algebra has two nonequivalent $1$- parameter family of deformations which are both infinitesimal and formal. We gave this deformation in an explicit form.
\vspace{.5cm}

{\bf \large Acknowledgements:}
The author would like to thank Professor A. Fialowski and Professor G. Mukherjee for their useful comments.

{\bf Ashis Mandal}\\
Stat-Math Unit, 
Indian Statistical Institute, \\
203 B.T. Road,
Kolkata- 700108, 
India.\\
e-mail: ashis$\_$r@isical.ac.in

\end{document}